\documentclass[12pt, oneside]{article}   	
\usepackage{geometry}                		
\usepackage{graphicx}				
\usepackage{amssymb, color}
\usepackage{tikz}
\usepackage[utf8]{inputenc}
\usepackage{booktabs}

\usepackage{array,multirow}

\usepackage{caption} 
\usepackage{hyperref}
\usepackage{hypcap} 

\title{An Estimation of Phase Transition}

\author{Shamsa Ishaq\thanks{Abdus Salam School of Mathematical Sciences
GC University, Lahore. 68-B, New MuslimTown, Lahore 54600, Pakistan. \texttt{Email:shamsi.ishaq@gmail.com}. 
This work was partially supported by funds from LMBA/ UMR 6205 and ASSMS. Shamsa Ishaq would like to acknowledge the support of LMBA and thanks to the University of Western Brittany-UBO, Brest for their kind hospitality where part of this work was completed.}
\,\,and\,\, Renaud Leplaideur \thanks{University of New Caledonia UNC. 145, Avenue James Cook - BP R4 98 851 - 
Noumea-Cedex New Caledonia. \texttt{Email:renaud.leplaideur@unc.nc}.}}

\date{}	
\usepackage{amsmath,amsthm, amssymb}
\usepackage{graphicx}
\usepackage{hyperref}
\usepackage{float}

\newtheorem{theorem}{Theorem}
\newtheorem{proposition}{Proposition}[section]
\newtheorem{lemma}[proposition]{Lemma}
\newtheorem{corollary}[proposition]{Corollary}

\newtheorem{definition}[proposition]{Definition}

\def\ie{{\em i.e.,\ }}

\def\N{{\mathbb N}}

\def\R{{\mathbb R}}

\def\K{{\mathbb K}}

\newcommand {\CB}{{\mathcal B}}
\newcommand {\CC}{{\mathcal C}}
\newcommand {\CD}{{\mathcal D}}
\newcommand {\CE}{{\mathcal E}}
\newcommand {\CF}{{\mathcal F}}

\newcommand {\CI}{{\mathcal I}}

\newcommand {\CM}{{\mathcal M}}

\newcommand {\CO}{{\mathcal O}}
\newcommand {\CP}{{\mathcal P}}

\newcommand {\CR}{{\mathcal R}}
\newcommand {\CS}{{\mathcal S}}
\newcommand {\CT}{{\mathcal T}}

\def\s{\sigma}
\def\l{\lambda}

\def\1{ {\hbox{{\it 1}} \!\! I} }

\def\be{\beta}
\def\de{\delta}
\def\ga{\gamma}
\def\om{\omega}

\def\v{\varphi}

\def\S{\Sigma}
\def\s{\sigma}
\def\G{\Gamma}

\def\8{\infty}
\def\p{\prime}

\def\8{\infty}
\def\llb{[\![}
\def\rrb{]\!]}

\renewcommand{\S}{\Sigma}
\usepackage{float,color}

\theoremstyle{definition}
\newtheorem{remark}{Remark}	
	
\begin{document}
\maketitle


\begin{abstract}
In \cite{Bruin_2015}, H. Bruin and R. Leplaideur studied a class of potentials such that the pressure function exhibit a phase transition at a parameter $\be_{c}>0$. This paper will prove that the transition in pressure function cannot appear within the interval $]0,2]$. 



\bigskip
\noindent
{\it AMS classification}: 37D35, 82B26, 37A60, 37B10, 68R15, 37F20.

\noindent
{\it Keywords}: thermodynamic formalism, freezing phase transition, Fibonacci word, equilibrium states.  
\end{abstract}


\section{Introduction.}
\label{introduction}

Let $T:X\rightarrow X$ be a continuous map defined on a compact metric space $X$. The topological pressure of a continuous function $\v:X\rightarrow \R$ satisfies the following variational principle:

\begin{equation}
\label{P-fct}
\CP(\be)=\sup_{\mu\in\CM_{T}(X)}\Big\{h_{\mu}+\be\int\v d\mu \Big\},
\end{equation}
 
 where $\CM(T)$ is the space of $T$-invariant probability measure, and $h_{\mu}$  is the Kolmogorov entropy (see \cite{walters2000introduction}). A measure attaining the supremum in \eqref{P-fct} is called an {\it equilibrium state}. During the 1970's it was shown that if the dynamical system $T$ is uniformly hyperbolic and the function $\v$ is regular enough (say H\"older), then there exists a unique equilibrium measure. Moreover, this measure has strong ergodic properties, and the function $\be\rightarrow P(\be\v)$ is real analytic (see, {\it e.g.} \cite{bowen2008equilibrium, keller1998equilibrium}). 
 
 Since the 80's, it has been a natural challenge to export this kind of results to systems with lighter hyperbolic  properties. The literature on that topic is very long. Note that there is still no general theory for Thermodynamic Formalism for non uniformly hyperbolic systems. 
 
 However, another natural question is to understand and study examples where the result fails and in particular, where/how/why it fails. This is the question of \emph{phase transition}. There are two natural strategies for that goal: either we study systems with lighter hyperbolicity or we study uniformly hyperbolic systems but deal with less regular $\v$. 
 
 For the first case, we mention, {\it e.g.} the Manneville Pomeau map or some unimodal maps being important examples (\cite{pomeau1980intermittent}). For the second case, Hofbauer in 1977 constructed a non H\"older function $\v$ on the full shift on two symbols for which its corresponding pressure function $\be\rightarrow P(\be\v)$ is not real analytic \cite{hofbauer1977examples}. 

 Roughly speaking, the main idea within this strategy of constructing systems exhibiting phase transitions is to consider a uniformly hyperbolic system and a continuous function which is H\"older in most of the space but not in a $T$-invariant subset. This subset is considered as ``the bad set''. In the case of Hofbauer's example, the bad set is a fixed point. This idea was later generalized by Markley and Paul (see \cite{markley1982equilibrium}) to include more general invariant sets (actually these could have positive entropy). More general manifestations of this idea can be found in the work of Climenhaga and Thompson (see \cite{climenhaga2013equilibrium,climenhaga2016unique}). It turns out that when the supremum in \eqref{P-fct} is attained in measures supported in the bad set where the function is not H\"older, then  phase transitions can occur. 
 
 We emphasize that both strategies (to weaken hyperbolicity or weaken regularity of $\v$) are not so disconnected. A link was done in  \cite{baraviera2012potential} between, on the one hand the Manneville-Pomeau map, and on the other hand the Hofbauer potential in the 2-full shift. The link was based on a \emph{renormalization operator} acting on potentials.
 This renormalization operator on potentials induced a renormalization on the dynamics, and led Bruin and Leplaideur to study the connection with substitutions in the full shifts (see \cite{Bruin_2013, Bruin_2015}). In these cases, the bad set is the natural invariant set (an attractor) for the considered substitution. The two classical Thue-Morse  and Fibonacci cases were studied. 
 
 In \cite{Bruin_2015}, the authors proved that the pressure function could exhibit a phase transition at the parameter $\be_c$ for the Fibonacci case, and pose the question to estimate the value of $\be_{c}$ (see, \cite[p.743]{Bruin_2015}). In this article, we will give a partial answer to the question that $\be_{c}\not\in ]0,2]$. 
 
 We emphasize that this results exhibit a new phenomenon, different from the phase transition in the Hofbauer case. Indeed (see below), a classical computation shows that adjusting well parameters, the phase transition in the Hofbauer case can occur for any value of $\be$ in $]1,+\8[$. Here, it must be bigger than 2, whatever the parameters are (our $\v$ looks like the Hofbauer potential). We explain this because in the Hofbauer case, the bad set is a single fixed point, whereas, for substitution, it is a \emph{quasi-crystal}. Even if it is minimal and with zero entropy, a quasi-crystal ``generates more chaos around it'' (in particular it does not have a local product structure) that seems to be the argument to prevent a phase transition at any temperature. 
 
 The lower bound that we obtain here may appear as a modest contribution to estimate the critical $\be_{c}$ for the Fibonacci substitution. However, we emphasize the total lack of tools in Ergodic Theory to study in depth chaos that happens at sub-exponential scale. Hence, our present work has to be seen as a work in progress in view to invent  and develop methods and  tools that we need to study Thermodynamic Formalism with zero entropy.

\subsection{Main setting and Result.}
\label{sec-main setting}

We denote $\S=\{0,1\}^{\N}$ is the one sided binary full shift. An element $x\in \S$, $x=x_0x_1x_2\cdots$ is called one-sided infinite word. The shift map is the map $\s:\S\rightarrow\S$ defined by, $\s(x_0x_1x_2\cdots)=x_1x_2x_3\cdots.$ A finite word is a finite sequence of 0 and 1. Let $u=u_0u_1\cdots u_{n-1}$ be a finite word, its length is $n$ and is denoted by $|u|$. We denote $\S^{+}$ the set of non-empty words, and $\epsilon$ is the empty word with $|\epsilon|=0$. 

The distance between $x_{0}x_{1}x_{2}\ldots $ and $y_{0}y_{1}y_{2}\ldots$ is given by $2^{-n(x,y)}$ where 
$$n(x,y)=\min\{k\in \N\cup\{+\8\},\ x_{k}\neq y_{k}\}.$$

A finite word $u=u_0u_1\cdots u_{n-1}$ is called a factor of a word $v$ if there exist words $r, s$ such that $v=rus$. The word $u$ is called
\begin{itemize}
\item a prefix of $v$ if $r=\epsilon$, 
\item a suffix of $v$ if $s=\epsilon$,
\item an inner factor of $v$ if $r\not=\epsilon$ and $s\not=\epsilon$.
\end{itemize}

The (infinite) {\it Fibonacci word},
$$\om= 0 1 0 01 010 01001 01001010 0100101001001 \cdots,$$ is the limit of sequence of finite Fibonacci words, defined inductively by $\om_{0}=0$, $\om_{1}=01$ and $$\om_{n+1}=\om_{n}\om_{n-1},\quad\quad n\geq 1.$$
The length of the word $|\om_{n}|=F_{n}$, where $\{F_n: n\in \N\}$ is the sequence of {\it Fibonacci numbers}, with the initial condition $F_{0}=1$, $ F_{1}=2$, and $F_{n+1}=F_{n}+F_{n-1}$. 
We denote by $L(\om)$ the set of factors of $\om$. 
The closure of the orbit, $$\K:= \overline{\{\sigma^{n}(\om) : n\in \N\}},$$ is the subshift associated with $\om$. The infinite Fibonacci word has remarkable combinatorial and dynamical properties see, for instance, \cite{arnoux1991representation, berstel1986fibonacci, cassaigne1997complexite, lothaire2002algebraic, zhi1994some}.\\

Given a finite word $u=u_0u_1\cdots u_{n-1}$, the corresponding {\it cylinder set} $[u]$ is defined by:
$$[w]=\{y\in \S : y_i=u_i\,,\,\,\forall\,\,\,0\leq i\leq n-1\}.$$ 
 A finite word $u$ is called a {\it return word} of the cylinder $[w]$, if the following conditions hold: 
 \begin{enumerate}\item $w$ is a prefix of $uw$, \item $w$ is not an inner factor of $uw$, \item $\min\bigl\{k: \s^k(ux)\in [w]\bigr\}=|u|$, for some $x\in [w]$. \end{enumerate}
We denote $\CR_{[w]}$ the set of return words to $[w]$.\\

For $x=x_{0}x_{1\ldots}\in \S$, we set
$$\de(x):=\max\{n:\,\forall\,\, k\le n, \quad \ x_{0}\ldots x_{k}\in L(\om)\}\le +\8.$$
Note that, $\de(x)=+\8$ if and only if $x\in \K$. In a similar way, if $u\in \S^{+}$ and $u\notin L(\om)$, then $\de(u)$ is the maximum length of common prefix of $u$ in $L(\om)$.
\\By definition:
\begin{equation*}
d(x,\K)=2^{-\de(x)-1}.
\end{equation*} 

We set $J=[000]$, then $\de(000)=2$ for all $x\in J$. Let $N$ be a positive integer such that $N>>2$ and $N-1\not= F_{n}-2$\footnote{The restriction on $N$ will be explained in Section \ref{sec-6}.}, for all $n\geq 2$. 

Let $A>0$, define a potential:
$$\varphi(x)=
\begin{cases}
 -\log\left(1+\frac1{\de(x)}\right) \text{ if }\de(x)\ge N,\\
 -A \text{ otherwise.}
\end{cases}
$$ 
For $\be\geq 0$, we set:
\begin{equation}
\label{eq-lambda}
\l_{\be}:=\sum_{u\in \CR_{J}}e^{\be S_{|u|}{\varphi(ux)}}
\end{equation}

where $x$ belongs to $J$. Note that, due to the form of $\v$, $\l_{\be}$ does not depend on the choice of $x$. 

(\cite{Bruin_2015}, Proposition 4.1) proved the existence of a parameter, $\be_{c}> 0$, such that $\l_{\be}<1$, for every $\be>\be_{c}$. From \cite[Theorem 4]{leplaideur2013local}, they got that a freezing phase transition occurs at $\be_{c}$: for any $\be>\be_{c}$, the unique equilibrium state for $\be\v$ is the unique invariant measure in $\K$.

Our main result is the following.

\begin{theorem}
\label{main-th}
Let $\be_c$ be the transition point of the pressure function in \cite[Theorem 2]{Bruin_2015}. Then $\be_{c}\not\in\, ]0,2]$.

 \end{theorem}
 
 The idea of the proof is to find a lower bound of $\l_{\be}$, that involves the identity $\zeta(\be-1)-2\zeta(\be)$, where $\zeta(\be)$ is the Riemann zeta function. Because for $\be>\be_{c}$ we have $\l_{\be}<1$, we will deduce from our lower bound that $\be_{c}$ has to be bigger than 2.

 \medskip
 We remind that for Hofbauer's case, at the transition point $\be_{c}$, equality $$\l_{\be_c}=e^{-A}\zeta({\be_c})=1,$$
 holds
 (see, \cite{leplaideur2013local}, Example 1). Therefore, the larger the value of $A$ is, the closer to 1 $\be_{c}$ is. Our result shows that this phenomenon cannot occur for the Fibonacci case. As we said above, we explain this by the fact that a quasi-crystal generates in its neighbourhood more chaos than a single fixed point (or even a periodic case). This is illustrated by the notion of accident defined in \cite{Bruin_2013, Bruin_2015}. 
 Therefore computing $\l_{\be_{c}}$ is not as easy as in Hofbauer's case (see,\cite[section 2]{ Bruin_2015}.  This is thus a challenge to develop tools in ergodic theory to allow a study in depth of different chaos at zero entropy scale. 
 
\subsection{Outline.}
\label{sec-outline}
The outline of the article is as follows:\\

In Section 2, we first recall some results of the Fibonacci words associated with special words and Rauzy graph. Section 2 is about the configuration of the Fibonacci word for a bispecial word (fixed) and its corresponding factor loops defined in \cite{arnoux1991representation}. The main result of this section is the proposition \ref{main-sec 2}, that plays a vital role to identify the identity $\zeta(\be-1)-2\zeta(\be)$ in the lower bound of $\l_{\be}$. At the end of section 2, we will define new words; Bicephalic words, and some basics of bicephalic words used in our final computation.\\
In Section 3, we will prove our main theorem, and find a lower found for $\l_{\be}$ in our final computation.

\section{Configuration of Fibonacci word.}
\label{sec-2}
Before describing the main section aim, we first will write about special words and the Rauzy graphs. We refer \cite{arnoux1991representation, cassaigne1997complexite,rauzy1982suites} for this section.
 
\subsection{Special words and Rauzy graphs.}
\label{sec-2.1}



The following definitions are from \cite{lothaire2002algebraic}.
\begin{definition}
The mirror of a word $u=u_0u_1\cdots u_{n-1}$, is the word $\tilde{u}=u_{n-1}u_{n-2}\cdots u_{1}$. A palindrome word is a word $u$ such that $\tilde{u}=u$. \\
For two words $u=u_0u_1\cdots u_{m-1}$ and $v=v_0v_1\cdots v_{n-1}$, the concatenation of $u$ and $v$ is the word $uv=u_0u_1\cdots u_{m-1}v_0v_1\cdots v_{n-1}$.
\end{definition}
We refer \cite{cassaigne1997complexite} for the following definition.
\begin{definition}
\label{def-special}
A finite word $u\in L(\om)$ is called a right extendable word in the Fibonacci word $\om$, if there exists a binary digit $c\neq\epsilon$ such that $uc\in L(\om)$. In that case, $c$ is called a right extension of $u$ in $\om$. 

Similarly, $u$ is called a left extendable word, and there exists $c\neq \epsilon$ such that $cu\in L(\om)$. In that case $c$  is called a left extension of $u$ in $\om$. 

A right extendable word is said to be {\it right special} if it has exactly two right extensions of length 1. 

A left extendable word  is said to be {\it left special} if it has exactly two left extensions of length 1. 

A word is called a bispecial word if it is right  and left special.
\end{definition}
We can visualise extensions (right and left) and special words through the Rauzy graph of  $\om$. The Rauzy graph of order $n$, $\G_{n}(\om)$ is a direct graph defined as follows: $\G_{n}(\om)$ has $n+1$ vertices labelled with elements of the set $L_{n}(\om)$, and there is an edge from vertex $u$ to vertex $v$ if, and only if, there exist two digits $a,b\in \{0,1\}$ such that $ua=bv\in L_{n+1}(\om)$, we label the edge by $u\overset{a}{\underset{b}{\longrightarrow}}v$.
For every $n\in \N$; $L_{n}(\om)$ has exactly one right special word, and exactly one left special word. All other words in $L_{n}(\om)$ have a unique right and left extension in $\om$. Therefore, the graph $\G_{n}(\om)$ has $n+2$ edges (see \cite[section 1]{arnoux1991representation}). We denote respectively by $i(u)$ and $o(u)$ the indegree and outdegree of a vertex $u$ in $\G_{n}(\om)$.
The Rauzy graph of order $n$ has exactly one vertex labelled with $A_{n}$ such that $o(A_{n})=2$, and has exactly one vertex labelled with $B_{n}$ such that $i(B_{n})=2$ (see \cite[section 1]{arnoux1991representation}).

Vuillon (see \cite[Section 4]{vuillon2001characterization}) observed that the Rauzy graph $\G_{n}(\om)$ is composed of three paths; the first two paths are as follows:
  
 \begin{equation}
 \label{eq-shape1}
\begin{cases}
 A_n\overset{a_{1}}{\underset{a_{k}}{\longrightarrow}}W_1{\overset{a_{2}}{\underset{a_{k-1}}\longrightarrow}}W_2\cdots {\overset{a_{k-1}}{\underset{a_{2}}\longrightarrow}}W_{k-1}{\overset{a_{k}}{\underset{a_{1}}\longrightarrow}}B_{n},\\ 
 
 A_n\overset{b_{1}}{\underset{b_{l}}{\longrightarrow}}U_1{\overset{b_{2}}{\underset{b_{l-1}}\longrightarrow}}U_2\cdots {\overset{b_{l-1}}{\underset{b_{2}}\longrightarrow}}U_{l-1}{\overset{b_{l}}{\underset{b_{1}}\longrightarrow}}B_{n},
\end{cases}
\end{equation}
 
$a_{1}\not= b_{1}$ and $k\not=l$, and the third path from $B_{n}$ to $A_{n}$ is 

\begin{equation}
\label{eq-shape2-2}
B_n\overset{c_{1}}{\underset{c_{s}}{\longrightarrow}}V_1{\overset{c_{2}}{\underset{c_{s-1}}\longrightarrow}}V_2\cdots {\overset{c_{s-1}}{\underset{c_{2}}\longrightarrow}}V_{s-1}{\overset{c_{s}}{\underset{c_{1}}\longrightarrow}}A_{n}.
\end{equation}
Furthermore, if $A_{n}=B_{n}$,  then $c_1c_2\cdots c_{s}=\epsilon$ in (\ref{eq-shape2-2}), and $\G_{n}(\om)$ consists of two disjoint loops around the vertex labelled by $A_{n}$(see \cite{arnoux1991representation}).

The following definition is from \cite[section 1]{arnoux1991representation}.
\begin{definition}
\label{def:FL}
Let $A_{n}=B_{n}$. Then the two factors $a_1a_2\cdots a_{k}$ and $b_{1}b_{2}\cdots b_{l}$ in (\ref{eq-shape1}) are called factor loops of the bispecial word $A_n$.
\end{definition}

Bispecial words are palindrome prefixes of $\om$ with length $F_{p}-2$, $p\geq 2$. Moreover, if a bispecial word has length $F_{p}-2$, then the associated factor loops have length $F_{p-1}$ and $F_{p-2}$\cite[Section 2]{Bruin_2015}. 

We fixed $p\geq2$, let $W$ be the bispecial word with $|W|=F_{p}-2$, and let $R_1,\,R_{2}$ are factor loops of $W$ such that $|R_1|=F_{p-1}$, $|R_2|=F_{p-2}$. We repeat that, in the Rauzy graph $\G_{|W|}(\om)$, the vertex $W$ is connected with two disjoint loops. From \cite[section 1]{arnoux1991representation} a finite word associated with a random walk; starts from $W$, end at $W$, after visiting $n$ times to vertex $W$, can be represented as, $$WR_{i_1}R_{i_2}R_{i_3}\cdots R_{i_{n-1}},\,\,\,\,i_j\in\{1,2\}\,,\,\,\,\,\,\,\text{for\,all} \,\,\,\,\,\,\,\,\,\,1\leq j\leq n-1.$$ For each $n\geq1$, here two questions arise naturally; 
\begin{enumerate}
\item What is the order(arrangement) of $R_{i_1}R_{i_2}R_{i_3}\cdots R_{i_{n-1}}?$
\item What is the $card\bigl\{R_{i_1}R_{i_2}R_{i_3}\cdots R_{i_{n-1}}: WR_{i_1}R_{i_2}R_{i_3}\cdots R_{i_{n-1}}\in L(\om)\bigr\}?$
\end{enumerate}
The aim of the section is to provide answers to questions mentioned above. \\
For each $n\geq 1$, let us define $$L_{n}(W)=\bigl\{R_{i_1}R_{i_2}R_{i_3}\cdots R_{i_n}: WR_{i_1}R_{i_2}R_{i_3}\cdots R_{i_n}\in L(\om), i_j=1,2,\,\,\,\forall \,\,1\leq j\leq n\bigr\}.$$
We  also define a block-digit code $\phi$, and its inverse $\phi^{-1}$ as follows,

\begin{equation}
\label{eq-phi}
\begin{cases}
 R_1\overset{\phi}{\longrightarrow}0,\\
 R_2\overset{\phi}{\longrightarrow}1,
\end{cases}\quad\quad\quad\quad
\begin{cases}
 0\overset{\phi^{-1}}{\longrightarrow}R_1,\\
 1\overset{\phi^{-1}}{\longrightarrow}R_2.
\end{cases}
\end{equation}
\\
The main result in this section is the following:
\begin{proposition}
\label{main-sec 2}
 \begin{enumerate}
For each $n\geq 1$ the following hold:
\item $L_n(\om)=\phi(L_{n}(W))$,
\item  $\phi^{-1}(L_n(\om))=L_{n}(W)$.
\end{enumerate}
\end{proposition}

\subsection{Proof of Proposition \ref{main-sec 2}.}
\label{sec-2.2}
The proof of proposition \ref{main-sec 2} is inspired by the technique used in \cite[section 1]{arnoux1991representation}.\\

 We repeat that $W$ is the bispecial word, and $|W|=F_{p}-2$, $p\geq2$. Moreover $R_1$, $R_2$ are factor loops of $W$, and $|R_1|=F_{p-1}$, $|R_2|=F_{p-2}$. \\Let $n\geq1$, we denote by $W^{(n)}$ be the bispecial word, with $|W^{(n)}|=F_{p+n}-2$, and $R_1^{(n)}$, $R_2^{(n)}$ are factor loops of $W^{(n)}$, such that $|R_1^{(n)}|=F_{p+n-1}, |R_2^{(n)}|=F_{p+n-2}$.

\begin{lemma}
\label{cor-F-1}
For each $n\geq 1$, the bispecial word $W^{(n)}$ is in the following form, 
\begin{equation}
\label{eq1}
W^{(n)}=WR_{1}^{(0)}R^{(1)}_1R^{(2)}_1\cdots R^{(n-1)}_1.
\end{equation}
Furthermore, the associated factor loops of $W^{(n)}$ are 

\begin{equation}
\label{eq-corF1}
R_1^{(n)}=R_2^{(n-1)}R_1^{(n-1)},\quad \text{and} \quad R_2^{(n)}=R_1^{(n-1)},
\end{equation}
with the intial condition $R_1^{(0)}=R_1$.

\end{lemma}

\begin{proof}
We will prove by induction. 
\begin{enumerate}
\item Step I: For $n=1$,
since $W$ is bispecial, therefore, the two disjoint loops around vertex $W$ in the Rauzy graph $\G_{|W|}(\om)$ is as follows:
$$ W \overset{a_{1}}{\underset{a_{k}}{\longrightarrow}}\cdots \overset{a_{k}}{\underset{a_{1}}{\longrightarrow}} W,\,\,\,\,\,\text{and} \,\,\,\,\, W \overset{b_{1}}{\underset{b_{l}}{\longrightarrow}}\cdots \overset{b_{l}}{\underset{b_{1}}{\longrightarrow}} W,$$ $a_1\not= b_1$ and $l\not=k$ (see \eqref{eq-shape1}). With no loss of generality, assume that $Wa_1$ and $a_1W$ are the labelled vertices in $\G_{|W|+1}(\om)$, such that $o(a_1W)=i(Wa_1)=2$. Since $W{a_1}\not=a_{1}W$, otherwise, $Wa_1$ is a bispecial word, that is not possible (see \cite[section 1, case II]{arnoux1991representation}). 

The third path in the graph $\G_{|W|+1}(\om)$ is as follows:

$$ Wa_1\overset{\alpha_{2}}{\longrightarrow}\cdots {\underset{\alpha_{2}}{\longrightarrow}} a_1W,$$(see \cite[section 1, case I]{arnoux1991representation}).
Again, $Wa_1a_2\not=a_2a_1W$, and $o(a_2a_1W)=i(Wa_1a_2)=2$, therefore, the third path from vertex $Wa_1a_2$ to vertex $a_2a_1W$ in $\G_{|W|+2}(\om)$ is as follows:
$$Wa_1a_2\overset{a_{3}}{\longrightarrow}\cdots {\underset{a_{3}}{\longrightarrow}}a_2a_1W.$$

Continuing the process, until we reach $\G_{m+(k-1)}(\om)$ such that, 
$$ Wa_{1}a_{2}\cdots a_{k-1} \overset{a_{k}}{\underset{a_{k}}{\longrightarrow}} a_{k-1}\cdots a_{2}a_{1}W.$$ By definition, $Wa_{1}a_{2}\cdots a_{k-1}a_{k}=a_{k}a_{k-1}\cdots a_{2}a_{1}W$ with $o(a_{k}a_{k-1}\cdots a_{2}a_{1}W)=i(Wa_{1}a_{2}\cdots a_{k-1}a_{k})=2$. Therefore, $Wa_{1}a_{2}\cdots a_{k-1}a_{k}$ is the bispecial word with length $F_{p+1}-2$ and $k=F_{p-1}$.

From \cite[section 1, case I, II]{arnoux1991representation}, we find directly $R_1^{(1)}=R_2R_1\quad and \quad R_2^{(1)}=R_1$.

\item Step II: Let $k\geq 1$ be given and suppose (\ref{eq1}) and (\ref{eq-corF1}) are true for $n=k$. Then, from Step I, $$W^{(k+1)}=W^{(k)}R_{1}^{(k)}=WR_{1}^{(0)}R^{(1)}_1R^{(2)}_1\cdots R^{(k-1)}_1R_{1}^{(k)}.$$ Moreover, from \cite[section 1, case I, II]{arnoux1991representation}, $R_1^{(k+1)}=R_2^{(k)}R_1^{(k)},\quad and \quad R_2^{(k+1)}=R_1^{(k)}$. Thus (\ref{eq1}) and (\ref{eq-corF1}) hold for $n=k+1$, and the proof of induction step is complete.
\end{enumerate}
\end{proof}
The following corollaries are from \cite{lothaire2002algebraic} and are the immediate consequences of Lemma \ref{cor-F-1}.
\begin{corollary}
\label{cor-PS}
Let $W^{(n)}$ be the bispecial word with length $F_{p+n}-2$. Then $W^{(k)}$ , for all $0\leq k\leq n$, are the palindrome prefixes and the palindrome suffixes of $W^{(n)}$.
\end{corollary}

\begin{corollary}
\label{cor-sp-fac}
Let $V$ be a purely left special word, and $|V|\geq |W|$. Then there exists a $k\geq 1$, such that, $V$ is a prefix of the bispecial word $WR_{1}^{(0)}R^{(1)}_1R^{(2)}_1\cdots R^{(k)}_1$, and the bispecial word $WR_{1}^{(0)}R^{(1)}_1R^{(2)}_1\cdots R^{(k-1)}_1$ is a prefix of $V$. 

Similarly, let $U$ be a purely right special word, and $|U|\geq |W|$. Then there exists a $k\geq 1$, such that, $U$ is a suffix of the bispecial word $WR_{1}^{(0)}R^{(1)}_1R^{(2)}_1\cdots R^{(k)}_1$, and the bispecial word $WR_{1}^{(0)}R^{(1)}_1R^{(2)}_1\cdots R^{(k-1)}_1$ is a suffix of $U$.
\end{corollary}

\begin{lemma}
\label{prop-decomp}
For every $n\geq 1$;

\begin{enumerate}
\item $\phi(R_1^{(n)})=\widetilde{\om_n},$
\item $\phi^{-1}\widetilde{(\om_{n})}=R_1^{(n)},$
\end{enumerate}

where $\phi$ is defined in \eqref{eq-phi}, and $\widetilde{\om_n}$ is the mirror of the $nth$ Fibonacci word.
\end{lemma}

\begin{proof}
The proof is done by induction.
\begin{enumerate}
\item \begin{enumerate}
\item For $n=0$, we have $\phi(R_{1})=0=\widetilde{\om_0}$.\\ For $n=1$, from \eqref{eq-corF1}, $\phi(R^{(1)}_{1})=\phi(R_{2}R_{1})=\phi(R_{2})\phi(R_{1})=10=\widetilde{\om_1}$. 
\item Assume that  $\phi(R_1^{(k)})=\widetilde{\om_k}$, for all value $k$ such that $0\leq k\leq n$. Then from\eqref{eq-corF1}, we get 
$$\phi(R_1^{(n+1)})=\phi(R_2^{(n)})\phi(R_{1}^{(n)})=\phi(R_1^{(n-1)})\phi(R_{1}^{(n)})=\widetilde{\om_{n-1}}\widetilde{\om_{n}}=\widetilde{\om_{n+1}}.$$\end{enumerate}

\item \begin{enumerate}
\item For $n=0$, we have $\phi^{-1}(0)=\phi^{-1}(\widetilde{\om_{0}})=R^{(0)}_{1}$. For $n=1$, $$\phi^{-1}(10)=\phi^{-1}(\widetilde{\om_{1}})=R^{(0)}_{2}R^{(0)}_{1}=R_{1}^{(1)}.$$  
\item Assume that, $\phi^{-1}\widetilde{(\om_{n})}=R_1^{(n)},$ for all $0\leq k\leq n$, then $$\phi^{-1}(\widetilde{\om_{n+1}})= \phi^{-1}(\widetilde{\om_{n-1}}\widetilde{\om_{n}})=\phi^{-1}(\widetilde{\om_{n-1}})\phi^{-1}(\widetilde{\om_{n}})=R_{1}^{(n-1)}R_{1}^{(n)}$$ From \eqref{eq-corF1}, we have $$\phi^{-1}(\widetilde{\om_{n+1}})=R_{2}^{(n)}R_{1}^{(n)}=R_{1}^{(n+1)}.$$
\end{enumerate}
This completes the proof.
\end{enumerate}

\end{proof}

\begin{lemma}
\label{lem-PBS}
Let $\xi_n$ be a palindrome prefix of $\om_{n}$ such that $\om_n=\xi_nxy$, where $xy=10$ if $n$ is odd and $xy=01$ if $n$ is even. Then $$\xi_n=\widetilde{\om_0}\widetilde{\om_1}\cdots \widetilde{\om_{n-2}},$$ for all $n\geq 2$.
\end{lemma}
\begin{proof}
We will proof by induction. First for $\xi_2=0=\widetilde{\om_0}$, and $\xi_3=010=\widetilde{\om_0}\widetilde{\om_1}$. 

Assume that $$\xi_k=\widetilde{\om_0}\widetilde{\om_1}\cdots \widetilde{\om_{k-2}},$$ for all value $k$ such that $0\leq k\leq n$. 
Let us set $\om_n=\xi_nxy$. Then $\om_{n-1}=\xi_{n-1}yx$ holds. 
Since $\om_{n+1}=\om_{n}\om_{n-1}$,  we get $\om_{n+1}=\xi_nxy \xi_{n-1}yx$. 

Hence, $$\xi_{n+1}=\xi_{n}xy \xi_{n-1}=\widetilde{\om_0}\widetilde{\om_1}\cdots \widetilde{\om_{n-2}}xy\xi_{n-1}.$$ 
Now, $\xi_{n-1}$ is a palindrome prefix, which yields $\widetilde{\om_{n-1}}=xy\xi_{n-1}$. This completes the proof.

\end{proof}

\begin{lemma}
\label{cor-code Bis}
\begin{enumerate}
\item Let $V=R_1^{(0)}R_1^{(1)}R_1^{(2)}R_1^{(3)}R_1^{(4)}\cdots R_1^{(k-2)}$, then $\phi(V)$ is the palindrome prefix of $\om_{k}$.
\item Let $v$ is a palindrome prefix of $\om_{k}$, then $\phi^{-1}(v)=R_1^{(0)}R_1^{(1)}R_1^{(2)}R_1^{(3)}R_1^{(4)}\cdots R_1^{(k-2)}$.
\end{enumerate}
 \end{lemma}
\begin{proof}
Its proof is an immediate consequence of lemma \ref{prop-decomp} and lemma \ref{lem-PBS}.
\end{proof}

{\bf Proof of Proposition \ref{main-sec 2}.}
We define

\begin{equation*}
\label{eq-decompose}
\Omega_{_W}=\prod\limits_{n\geq 0}R_1^{(n)}=R_1^{(0)}R_1^{(1)}R_1^{(2)}\cdots R_1^{(n)}\cdots,  
\end{equation*}
is the concantenation of sequence of words $\{R_1^{(n)}\}$.

From the factorization, $\om=\prod\limits_{n\geq 0}\widetilde{\om_n}$ (see \cite[proposition 17]{fici2015factorizations}), and lemma \ref{prop-decomp}, we get the following,

\begin{equation}
\label{eq-Omega}
\begin{cases}
 \phi(\Omega_{_W})=\om,\\
 \phi^{-1}(\om)=\Omega_{_W}.
\end{cases}\end{equation}

Now for each $n\geq 1$, define the following set,
$$L_{n}(\Omega_{_W})=\bigl\{R_{i_1}R_{i_2}R_{i_3} \cdots R_{i_n}: WR_{i_1}R_{i_2}R_{i_3}\cdots R_{i_n}\in L(\om), i_j=1,2\,\,\forall \,1\leq j\leq n\bigr\},$$ 
then from \eqref{eq-Omega}, we have the following
\begin{equation*}
\label{eq-omega}
\begin{cases}
 L_n(\om)=\phi(L_{n}(\Omega_{_W})),\\
 \phi^{-1}(L_n(\om))=L_{n}(\Omega_{_W}).
 
\end{cases}\end{equation*}
Note that, $L_{n}(W)=L_{n}(\Omega_{_W})$ (see, end of section \ref{sec-2.1}). Later, throughout the article, we will use notation $L_{n}(\Omega_{_W})$.

 \begin{remark}
\label{rem-1}
For every $n\geq1$, the set $L_{n}(\Omega_{_W})$ may have elements with unequal length. Indeed, if $U, V\in L_{n}(\Omega_{_W})$ with $|\phi(U)|_{0}\footnote{$|\phi(U)|_{c}$ is the number of occurrences of a binary digit $c$ in $\phi(U)$.} =|\phi(V)|_{0}$, and $|\phi(U)|_{1} =|\phi(V)|_{1}$ then $|U|=|V|$, otherwise $|U|\not=|V|$. More precisely, let $n=3$  $$L_{3}(\Omega_{_W})=\bigl\{R_{1}R_{2}R_{1}, R_{1}R_{1}R_{2}, R_{2}R_{1}R_{1}, R_{2}R_{1}R_{2}\bigr\},$$ then
\begin{itemize}
\item[(i)] $\phi(R_{1}R_{2}R_{1})=010$,
\item[(ii)] $\phi(R_{1}R_{1}R_{2})=001$,
\item[(iii)]  $\phi(R_{2}R_{1}R_{1})=100$ ,
\item[(iv)] $\phi(R_{2}R_{1}R_{2})=101$.
\end{itemize}

The words 010, 001 and 100 have an equal number of $0$ and $1$. Therefore, the three words $R_{1}R_{2}R_{1}, R_{1}R_{1}R_{2},$ and $R_{2}R_{1}R_{1}$ in $L_{3}(\Omega_{_W})$ have a length equal to $F_{p+1}$, and the word $R_{2}R_{1}R_{2}$ has length $F_{p}+F_{p-2}$.
\end{remark}

 
  
  

\subsection{Bicephalic Words.}
\label{sec-2.3}
The aim of this section is to define bicephalic words, and to give some of their basic properties. The main result of this section is proposition \ref{prop-CB}.

\begin{definition}
\label{def-bicephalic}
A finite word $U$ in $L(\om)$ is called a bicephalic word for a pair of the special word $(W,W^{\p})$, if $U=WU^{\p}W^{\p}$, for some $U^{\p}\in L(\om)$. A bicephalic word $U$ is called an identical bicephalic if $W=W^{\p}$.
\end{definition}


\begin{lemma}
\label{lem-MB}
The mirror of a bicephalic word is bicephalic.
\end{lemma}
\begin{proof}
Mirror of a special word is itself special (see \cite[section 4.10.3]{berthe2010combinatorics}), therefore, $\widetilde{U}=\widetilde{W^{\p}}\widetilde{U^{\p}}\widetilde{W}$ is bicephalic for pair $(\widetilde{W^{\p}}, \widetilde{W})$.
\end{proof}

Later, we are more interested in identical bicephalic for a bispecial word.

We recall that $W$ is a bispecial word with length $F_p-2$, and $R_{1}$, $R_{2}$ are the corresponding factor loops of $W$ with length $F_{p-1}$ and $F_{p-2}$ (respectively). For every $n\geq 1$, we denote by $U_{_W}(n)$ an identical bicephalic for $W$, and $n$ is the number of occurrence of $W$ in $U_{_W}(n)$. Note that there are several such words. One of the purpose here is to make precise the number of possibilities. 

\begin{lemma}
\label{prop-decomp-bicep}
A bicephalic word $U_{_W}(n)$ can be factorized as;
$$U_{_W}(n)=WV,$$ for some $V\in L_{n-1}(\Omega_{_{W}})$.
\end{lemma}
\begin{proof}
It is an immediate consequence from the formation of Rauzy graph $\G_{|W|}(\om)$ (see, \cite[Section 1, Case II]{arnoux1991representation}), and proposition \ref{main-sec 2}.
\end{proof}

From proposition \ref{main-sec 2}, we conclude that, for every $V\in L_{n-1}(\Omega_{_{W}})$, the word $WV\in L(\om)$.

For each $n\geq 1$, let us denote $\CB_{_W}(n)=\{ WV: V \in L_{n-1}(\Omega_{_{W}})$, is the set of all possible $U_{_W}(n)$. The following proposition is crucial to our main computation.

 \begin{proposition}
 \label{prop-CB}
 For each $n\geq 1$, the set $\CB_{_W}(n)$ has the following properties:
 \begin{itemize}
\item[(i)] $card\bigr(\CB_{_W}(n)\bigl)= n$,
\item[(ii)] If $n-1=F_m-2$ for some $m\geq2$, then $\CB_{_W}(n)$ contains exactly one special word, \ie a bispecial word,
\item[(iii)] If $n-1\not=F_m-2$  for all $m\geq2$, then $\CB_{_W}(n)$ contains exactly two special words.
\end{itemize}
\begin{proof}
\begin{itemize}
\item[(i)] From lemma \ref{prop-decomp-bicep}, and proposition \ref{main-sec 2}, we conclude that $$card\bigr(\CB_{_W}(n)\bigl)=card\bigr(L_{n-1}(\Omega_{_{W}})\bigl)=n.$$

\item[(ii)] If $n-1=F_m-2$, then from proposition \ref{main-sec 2}, there is exactly one word $V\in L_{n-1}(\Omega_{_W})$ such that $\phi(V)$ is a palindrome prefix of $\om$. From lemma \ref{cor-code Bis}, there exists a $k\geq 1$ such that $V=R^{(0)}_{1}R^{(1)}_{1}\cdots R^{(k-1)}_{1}$. Consequently, from lemma \ref{cor-F-1}, the word $WV$ is a bispecial word.
\item[(iii)] If  $n-1\not=F_m-2$, then from proposition \ref{main-sec 2}, there is exactly one word $U\in L_{n-1}(\Omega_{_W})$ such that $\phi(U)$ is a prefix of $\om$. Since prefixes of $\om$ are left special words (see  \cite[proposition 4.10.3]{berthe2010combinatorics}). Moreover, $U\not= R^{(0)}_{1}R^{(1)}_{1}\cdots R^{(k-1)}_{1}$, for all $k\geq 1$, otherwise $\phi(U)$ is a palindrome prefix (see lemma \ref{cor-code Bis}), and hence a bispecial word. This implies, $|\phi(U)|=n-1=F_m-2$, for some $m\geq 2$ (see proposition \ref{main-sec 2}), that is a contradiction. Therefore, $\phi(U)$ is a purely left special word(not a bispecial word). Then, equalities \eqref{eq-Omega}, imply that $U$ is a prefix of $\Omega_W$, and $U\not= R^{(0)}_{1}R^{(1)}_{1}\cdots R^{(k-1)}_{1}$, for all $k\geq 1$, therefore there exists a $m\geq 1$ such that $U$ is a prefix of $R^{(0)}_{1}R^{(1)}_{1}\cdots R^{(m-1)}_{1}$. The word $WR^{(0)}_{1}R^{(1)}_{1}\cdots R^{(m-1)}_{1}$ is a bispecial word (see lemma \ref{cor-F-1}), and has prefix $WU$. Since prefixes of a bispecial word are left special words (see corollary \ref{cor-sp-fac}), $WU$ is a left special word. The word $WU$ is a purely left special word, otherwise from lemma \ref{cor-F-1}, $U=R^{(0)}_{1}R^{(1)}_{1}\cdots R^{(l-1)}_{1}$ for some $l\leq m$, and that gives a contradiction. \\ From lemma \ref{lem-MB}, the word $\widetilde{WU}$ is an identical bicephalic. Moreover, the mirror of a purely left special word is a purely right special word(see  \cite[proposition 4.10.3]{berthe2010combinatorics}). This completes the proof.
\end{itemize}
\end{proof}
 \end{proposition}
 


\section{Proof of Theorem \ref{main-th}.}
\label{sec-3}
In this section, we will prove theorem \ref{main-th}. We reemploy some vocabulary and tools given in \cite{Bruin_2015}.
\subsection{$FE$ and $EF$-Transition.}
\label{sec-3.1}

From section \ref{sec-main setting}, we recall that $J=[000]$ and $\CR_J$ is the set of all return words to the cylinder $J$. Let $y\in J$, then there is a $u\in \CR_J$ with length $n$ such that $y=ux$, for some $x\in J$. We denote by
$\CO(y)=\{\s^{i}(y): n\in \N\}$ is the orbit of $y$. The Birkhoff sum for $y$ in \eqref{eq-lambda} is as follows: $$\CS_{|u|}\varphi(y)=\sum_{i=0}^{|u|-1}\varphi(\s^{i}(y)),$$ that depends on the subset $\CO_{u}(y)=\{\s^{i}(y): 0\leq i \leq |u|-1\}$, of the orbit set $\CO(y)$. Therefore, we will focus on the elements of set $\CO_{u}(y)$.\\

Concerning $N$(see section \ref{sec-main setting}), let us divide the full shift $\S$ into the two zones. The first zone is the free zone $\S_{\CF}:=\{x\in \S :\de(x)<N\}$, and the second zone is the excursion zone $\S_{\CE}:=\{x\in \S :\de(x)\geq N\}.$ 
Note that the cylinder $J$ is a subset of the free zone and the subshift $\K$ is a subset of the excursion zone. 
 
 Let us fixed a return word $u$, and $|u|=n$, $y=ux$.

\begin{definition}
\label{def-FE time}
Let $k\in \llb 0,n-1\rrb$. 
\begin{itemize}
\item[$($i$)$] If $\de(\s^{k}(y))<N$, we say that $k$  is a free time of $u$. 
\item[$($ii$)$] Otherwise, we say that $k$ is an excursion time of $u$. 
\end{itemize}
\end{definition}


\begin{definition} 
\label{def-accident}
An integer $k\in \llb 1,n-1\rrb $ is called an accident time in the orbit $\CO(y)$ if, and only if,  $$\de(\s^{k}(y))>\de(\s^{k-1}(y))-1,$$ holds. 
\end{definition}

We set $\de(\sigma^{k}(y))=d=|v|$, $d$ is called the depth of the accident appears at time $k$ in orbit $\CO(y)$ (see \cite{bedaride2015thermodynamic}).\\

Note that, if $k\in \llb 1,n-1\rrb $ is not an accident time, then 
\begin{equation}
\label{eq-no-acc}
\de(\sigma^{k}(y))=\de(\sigma^{k-1}(y))-1.
\end{equation} 
For more detail of the accident, we refer \cite{bedaride2015thermodynamic} and \cite{Bruin_2015}.\\

\begin{proposition}[see \cite{Bruin_2015}]
\label{prop-accident bispecial}
Let $k\in \llb 1,n-1\rrb$ be an accident time and $\delta(\sigma^{k-1}(y))=p$. Then the factor of $u$ appearing at the position $k$ with length $p-1$ is a bispecial word. 
\end{proposition}

As a consequence of definition \ref{def-accident}, elements of the set $\CO_{u}(y)$ roam in the free zone and the excursion zone. A significant transition in orbit appears at two times:

\begin{enumerate}
\item when $\s^{i}(ux)\in \S_{\CF}$, and $\s^{i+1}(ux)\in \S_{\CE}$,
\item when $\s^{j}(ux)\in \S_{\CE}$, and $\s^{j+1}(ux)\in \S_{\CF}$,
\end{enumerate}
for some $i,j\in \llb 1,n-1\rrb $.
 Indeed, our object in this section is to observe words in the language $L(\om)$; associated with returning $u$, when these two major transitions (mentioned above) appear in orbit.
\begin{definition}
\label{FE-transition}
An integer $k\in \llb 1,n-1\rrb $ is called FE-transition time if, and only if, $$\de(\s^{k}(ux))\geq N,\quad and \quad \de(\s^{k-1}(ux))\leq N-1,$$ \ie the orbit enters in excursion zone right after the free zone.
\end{definition}

The following proposition appears in \cite{Bruin_2015}, is about $FE$-transition.

\begin{proposition}
\label{prop-F to E}
Let $u$ be a return word of length $n$. Let $k\in \llb 1,n-1\rrb$ be a free time, such that $k+1$ is an excursion time for $ux$. Set $\de(\s^{k}(ux))=p\le N-1$.
Then, 
\begin{itemize}
\item[$($i$)$] $u_{k+1}\ldots u_{k+p-1}$ is a bispecial word with length $\leq N-2$.
\item[$($ii$)$] $u_{k}\ldots u_{k+p-1}$ is not a right special word. 
\end{itemize}

\end{proposition}

Since $FE$-transition time is an accident time (see, definitions \ref{def-accident} and \ref{FE-transition}), and from propositions \ref{prop-accident bispecial}, \ref{prop-F to E}, a bispecial word with length $\leq N-2$ appears in $u$ at $FE$-transition time. Concerning these results, we have the following definition.

\begin{definition}
\label{def-entry}
 We denote by:$$\CI=\{w\in L(\om): w\,\,is\,\,\,bispecial\,\,word,\,\,and\,\,|w|\leq N-2\}.$$  An element $w\in\CI$ is called an \emph{Entry Word}. 
\end{definition}
Note that, any $FE$-transition time in orbit associated with some entry word. 

\begin{definition}
\label{EF-transition}
An integer $k\in \llb 1,n-1\rrb $ is called EF-transition time if, and only if, $$\de^{k-1}(ux)\geq N,\quad and \quad \de^{k}(ux)\leq N-1,$$ \ie the orbit enters in the free zone right after the excursion zone.
\end{definition}

\begin{proposition}
\label{prop-E to F}
Let $u$ be a return word of length $n$. Let $k\le n-1$ be an excursion time, such that $k+1$ is a free time. Set $\s^{k}(ux)=z=z_{0}z_{1}\ldots$, then the following hold:
\begin{itemize}
\item[$($i$)$] $\de(z)=N$, $\de(\s(z))=N-1$,
\item[$($ii$)$] $z_{1}\ldots z_{N-1}$ is not a right special word.
\end{itemize}

\end{proposition}
\begin{proof}
\begin{description}
\item \textit{(i)}\,\, From definition \ref{def-FE time}, $\de(z)\ge N$, and $\de(\s(z))\le N-1$. Moreover,  $\de(\s(z))\not\ge \de(z)-1$. Therefore, $k+1$ is not an accident time (see definition \ref{def-accident}). From \eqref{eq-no-acc}, we have $\de(z)-1=\de(\s(z))$. Now if $\de(z)\geq N$, then $\de(\s(z))\geq N-1$, but $\de(\s(z))\leq N-1$. Therefore, $\de(\s(z))=N-1$, and $\de(z)=N$ (see \eqref{eq-no-acc}). This completes $(i)$.
\item\,\, \textit{(ii)} Since $\de(z)=|z_{0}\ldots z_{N-1}|$, and the word $z_{0}\ldots z_{N-1} \in L(\om)$. Equality $\de(\s(z))=N-1$ yields that $z_{1}\ldots z_{N-1}\in L(\om)$, and $z_{1}\ldots z_{N-1}z_{N}\not\in L(\om)$. Therefore, the word $z_{1}\ldots z_{N-1}\in L(\om)$ is not a right special word, otherwise, for any choice of $z_{N}\in\{0,1\}$, we always find $z_{1}\ldots z_{N-1}z_{N}$ belongs to $L(\om)$, and this yields to a contradiction that $\de(\s(z))>N-1$. 
\end{description}\end{proof}

We point out;  if $z=z_{0}\ldots z_{N}\ldots\in \S$ is such that $z_{0}\ldots z_{N-1}$ belongs to $L(\om)$, but $z_{1}\ldots z_{N}$ does not, then, $\de(z)=N$ and $\de(\s(z))=N-1$.\\

We assume that $v$ is a prefix of $\s(z)$ with $|v|=N-1$. From proposition \ref{prop-E to F}, $v$ is not the right special word. We observe that an $EF$-transition time is linked with a finite word with length $N-1$, which is not a right special word.

\begin{definition}
\label{def-exit}
We denote by: $$\CO=\{v\in L(\om): v\,\,is\,\,not\,\,a\,\,right\,\,special \,\,word,\,\,and\,\,|v|= N-1\}.$$  An element $v\in\CO$ is called an \emph{Exit Word}. 
\end{definition}

Since $|L_{N-1}(\om)|=N$, there are possibly $N-1$ exit word in the language $L(\om)$.


\subsection{Main computation.}
\label{sec-3.2}
We remind that $J=[000]$ and $\CR_J$ is the set of all return word to the cylinder $J$ (see section \ref{sec-main setting}). Let $W$ be a bispecial word with length $F_{p}-2 \geq N$, and let $R_1$ and $R_2$ are its corresponding factor loops with length $F_{p-1}$, $F_{p-2}$ respectively (see section \ref{sec-2}).

\medskip
We consider   the set $\CD_{W}$ of all trajectories satisfying the following conditions:
\begin{enumerate}
\item $y\in J$ and $y=ux$, for some $u\in \CR_{J}$, for some $x\in J$.
\item $\de(\s(y))\geq N$, \ie $1$ is the $FE$-transition time.
\item Only two accident appear in $\S_{\CE}$.
\item Assume that accidents in $\S_{\CE}$ appear at time $l,l^{\p}\in \llb 1,n-1\rrb $, $l\leq l^{\p} $, then $\de(\s^{l}(ux))\in [W]$, and $\de(\s^{l^{\p}}(ux))\in [W]$, \ie $W$ is the bispecial word appears at time $l,l^{\p}$ in the return word $u$(see proposition \ref{prop-accident bispecial}).
\item The exit word at $EF$-transition time is the purely left special word $v$ with length $N-1$\footnote{$N-1\not= F_s-2$, for some $s\geq 2$.}.
\item $l^{\p}$ is the last accident time in the set $\CO_{u}(y)$.
\end{enumerate}

Since there is only one bispecial word with length $F_p-2$, therefore for two bispecial words $W$ and $W^{\p}$, and $|W|\not=|W^{\p}|$, we always have $\CD_{W}\cap \CD_{W^{\p}}$ (see proposition \ref{prop-accident bispecial}).\\

From the above conditions (1)-(6), any $y\in \CD_{W}$ can be decomposed as follows:
\begin{equation*}
\label{eq-y-decomp}
y=000\cdots W\cdots W\cdots va x,
\end{equation*}
where $a\in\{0,1\}$ such that $va\not\in L(\om)$ (see definition \ref{def-exit}) and $x\in J$.

Furthermore, the corresponding return words of the trajectories in $\CD_{W}$ can be write as:

\begin{equation*}
\label{eq-r-decomp}
u000=000u^{\p}va000\footnote{If $a=0$, then $u000=000u^{\p}va00$.}
\end{equation*}

The word $u^{\p}$ is an inner factor of $u$, and the word $00u^{\p}va$ is the excursion part of the return word $u$. \\

Let $U=00u^{\p}vax$ for some $x\in J$. Then $$\de(\s^{i}(U))\geq N\,,\,\,\,\,\,\,\,\,\,\,\,\,\,\,\,\,\,\forall\,\,\,\,0\leq i \leq |u^{\p}|+1$$ and $$\de(\s^{ |u^{\p}|+2}(U))=N-1=|v|$$ This yield that  $|u^{\p}|+2$ is the $EF$-transition time.

Let us denote, $d_0,d_1$ and $d_2$ be the depts of the accidents appear at time $1, l$ and $l^{\p}$ respectively (see definition \ref{def-accident}). We take $U^{(0)}$, $U^{(1)}$ and $U^{(2)}$ are words in $L(\om)$ such as:

\begin{itemize}
\item[(i)] $\de\bigr(\s(Ux)\bigl)=d_0=|U^{(0)}|$
\item[(ii)] $\de\bigr(\s^{l}(Ux)\bigl)=d_1=|U^{(1)}|$
\item[(iii)] $\de(\s^{l^{\p}}\bigr(Ux)\bigl)=d_2=|U^{(2)}|$
\end{itemize}

The following lemma illustrate Figure \ref{Figure 1} (see \cite[Figure 2]{Bruin_2015}):

\begin{lemma}
\label{lem-dept-bicephalic}
The words $U^{(0)}$, $U^{(1)}$ and $U^{(2)}$ are non-special bicephalic words.
\end{lemma}

\begin{proof}
By the formation of the trajectories in $\CD_{W}$, the words can be decomposed as $U^{(0)}=0X^{(0)}W$, $U^{(1)}=WX^{(1)}W$ and $U^{(2)}=WX^{(1)}v$.
\end{proof}

\begin{figure}[htbp]
\begin{center}
\includegraphics[scale=0.6]{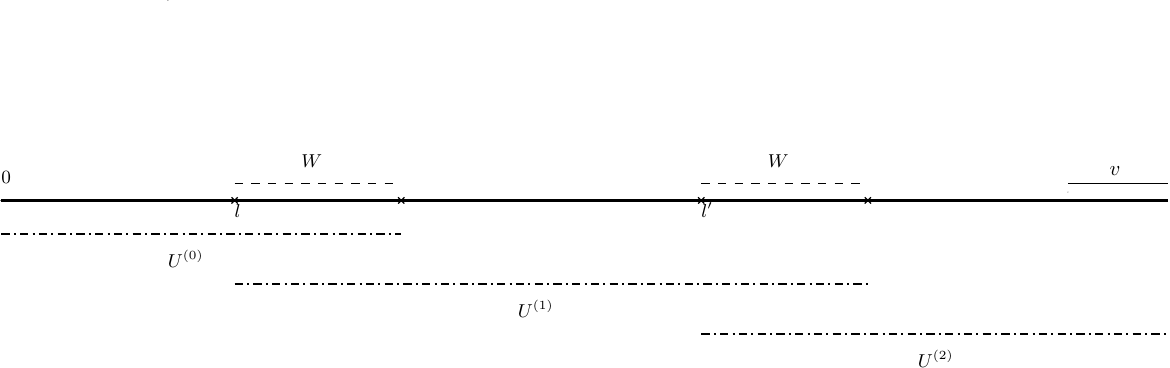}
\caption{Decomposition of excursion part}
\label{Figure 1}
\end{center}
\end{figure}

We denote by, 
\begin{equation}
\label{eq-DW}
\CC_{W}(\be):=\sum_{y=ux\in \CD_{W}}e^{\be S_{|u|}{\varphi(ux)}}
\end{equation}

Then from \eqref{eq-lambda}, we have the following inequality:

\begin{equation}
\label{eq-2}
\l_{\be}\geq\sum_{|W|\geq N}\CC_{W}(\be).
\end{equation}
 
 From \cite[Proposition 4.1]{Bruin_2015}, we observe that;

 $$\CC_{W}(\be)=e^{-NA}\sum\limits_{\substack{U^{(0)}, U^{(1)}\\\,\, U^{(2)}\,\,possible}}{\Bigl(\frac{|W|+1}{|U^{(0)}|+1}\Bigr)}^{\be}{\Bigl(\frac{|W|+1}{|U^{(1)}|+1}\Bigr)}^{\be}  {\Bigl(\frac{N+1}{|U^{(2)}|+1}\Bigr)}^{\be}.$$

The words $U^{(0)}$, $U^{(1)}$ and $U^{(2)}$ are in Lemma \ref{lem-dept-bicephalic}.

We denote the following identities:
\begin{enumerate}
\item $\CT_{0}(\be)=\sum_{k\geq 1}\sum\limits_{\substack{U^{(0)}(k)\\\,\, U^{(0)}(k)\,\,non\,special}}\Bigl({{|W|+1}\over {|U^{(0)}(k)|+1}}\Bigr)^{\be},$
\item $\CT_{W}(\be)=\sum_{k\geq 1}\sum\limits_{\substack{U^{(1)}(k)\in \CB_{W}(k)\\\,\, U^{(1)}(k)\,\,non\,special}}\Bigl({{|W|+1}\over {|U^{(1)}(k)|+1}}\Bigr)^{\be}$, 
\item $\CT_{v}(\be)=\sum_{k\geq 1}\sum\limits_{\substack{U^{(2)}(k)\\\,\, U^{(2)}(k)\,\,non\,special}}\Bigl({{N+1}\over {|U^{(2)}(k)|+1}}\Bigr)^{\be}.$

\end{enumerate}
Then, 
\begin{equation}
\label{eq-9}
\CC_{W}(\be)=e^{-NA}\CT_{0}(\be)\CT_{W}(\be)\CT_{v}(\be).
\end{equation}

\subsubsection{Lower bound for $\CT_{0}(\be)$ and $\CT_{W}(\be)$.}
\label{sec-3.2.1}

In Identity $\CT_{0}(\be)$ the sum runs over all possible non-identical bicephalic (non-special) for $(0,W)$. For every $k\geq 1$, the word $U^{(0)}(k)$ is a non-identical bicephalic for $(0,W)$ (see Lemma \ref{lem-dept-bicephalic}), and $k$ is the number of occurrence of $W$. Therefore, $$U^{(0)}(k)=U^{(0)}(1) R_{i_1}R_{i_2}\cdots R_{i_{k-1}},$$ and $R_{i_1}R_{i_2}\cdots R_{i_{k-1}}\in L_{k-1}(\Omega_{_W})$. The word 0 is a prefix of $W$ (see corollary \ref{cor-PS}), therefore, $|U^{(0)}(1)|\leq |WR_1|\leq F_{p}-2+F_{p-1}$, and $$|U^{(0)}(k)|\leq F_{p}-2+k F_{p-1}.$$
 
Consequently:
$$\CT_{0}(\be) \geq \sum_{k\geq 1}\Bigl({{F_p-1}\over {(F_p-1)+(k)(F_{p}-1})}\Bigr)^{\be},$$
this implies,
\begin{equation}
\label{eq-T0}
\CT_{0}(\be)\geq \zeta(\be)-1.
\end{equation}

\begin{lemma}
\label{lem-1}
Let $\be>2$, then 
\begin{equation}
\label{eq-zeta}
\CT_{W}(\be)\geq \zeta(\be-1)-2\zeta(\be),
\end{equation}
where $\zeta(\be)$ is the Riemann zeta function.
\end{lemma}

\begin{proof}
The sum in $\CT_{W}(\be)$ runs over all possible identical bicephalic (non-special) words for $W$. Furthermore, $U^{(1)}(k)=U_{_W}(k)\in \CB_{W}(k)$ (see lemma \ref{prop-decomp-bicep}).
For each $k\geq1$, we have 
$U^{(1)}(k)=WR_{i_1}R_{i_2}\cdots R_{i_{k-1}}$, for some $R_{i_1}R_{i_2}\cdots R_{i_{k-1}}\in L_{k-1}(\Omega_{_W})$, and the following inequality holds,
\begin{equation}
\label{eq-4-3}
 |U^{(1)}(k)|+1\leq F_{p}-1+(k-1)F_{p-1}.
\end{equation} 
From proposition \ref{prop-CB}, if $k-1=F_{s}-2$, for some $s\geq 2$, then the set $\CB_{W}(k)$ contains exactly $k-1$ non-special words, otherwise $\CB_{W}(k)$ contains exactly $k-2$ non-special words. Therefore from \eqref{eq-4-3}


$$\CT_{W}(\be)\geq  \sum_{k\geq 1}(k-2)\Bigl({{F_{p}-1}\over {(F_{p}-1)+(k-1)(F_{p}-1)}}\Bigr)^{\be}=\sum_{n\geq 1}{(k-2)\over k^{\be}}.$$ Hence we obtained \eqref{eq-zeta}.
\end{proof}

\subsubsection{Lower bound for $\CT_{v}(\be)$.}
\label{sec-3.2.2}

The word $v$ is a purely left special(not a bispecial) word with length $N-1$. Therefore there is exactly one right special word with length $N-1$, say $v^{\p}$(see proposition 4.5.8 {\cite{berthe2010combinatorics}). Furthermore, in the Rauzy graph $\G_{_{N-1}}(\om)$, the factor associated with the path from $v$ to $v^{\p}$ is non-empty (see lemma \ref{cor-F-1} (step I)). Let us take the path from $v$ to $v^{\p}$ is as follows;

\begin{equation}
\label{eq-1}
v\overset{\xi_{1}}{\longrightarrow} \cdots {\overset{\xi_{t}}\longrightarrow}v^{\p}, 
\end{equation}
and $\xi_1\xi_2\cdots \xi_{t}\not= \epsilon$. 
Let $P$ be the least positive integer such that $\ga^{P}\geq N$, and $Q$ be the greatest positive integer such that $\ga^{Q}\leq N-1$. Let $W^{P}$ be the bispecial word with length $F_P-2$ and $W^{Q}$ be the bispecial word with length $F_Q-2$. From corollary \ref{cor-sp-fac}, $W^{Q}$ is a prefix of $v$, and $v$ is a prefix of $W^{P}$. Also, $W^{P}=v\xi_1\xi_2\cdots \xi_{t}$ (see proof of lemma \ref{cor-F-1} (step I)). Let us denote $R_1^{P}$ and $R_2^{P}$ be the factor loops of $W^{P}$ with length $F_{P-1}$ and $F_{P-2}$ respectively. 

\begin{lemma}
\label{lemWP}
The word $W^{P}R_{i_1}^{P}R_{i_2}^{P}\cdots R_{i_{h-1}}^{P}$, $i_j=1,2$, $1 \leq j \leq h-1$ has exactly $h$ occurrences of $v$.
\end{lemma}
\begin{proof}
It is enough to prove that $W^{P}R_i^{P}$, $i=1,2$, has exactly two occurrences of $v$. For this let $\G_{_{N-1}}(\om)$ be the Rauzy graph of order $N-1$. There are precisely two loops from vertex $v$ to $v$ in such a way that factors associated with these two loops share a common prefix $\xi_1\xi_2\cdots \xi_{t}$, [see, for instance, \cite{arnoux1991representation,cassaigne1997complexite}], and have unequal length. We denote $S_1$ and $S_2$ corresponding to the factors mentioned above, and $|S_1|\geq |S_2|$. There is no bispecial word with length in $[[N-1, F_{P}-2]]$, except $W^{P}$. Therefore from \cite[section I, case I]{arnoux1991representation}, we have $|S_1|=|R_1^{P}|$, and $|S_2|=|R_2^{P}|$. The word $v$ is a prefix of $W^{P}$, therefore in the word $W^{P}R_i^{P}$, for $i=1,2$, $v$ has at least two occurrences, with the occurrences of $W^{(P)}$. We suppose that $v$ has exactly three occurrences in $W^{P}R_i^{P}$, then $$W^{P}R_i^{P}=vS_{j_1}S_{j_2}\xi_1\xi_2\cdots \xi_{t}.$$ Since $W^{P}=v\xi_1\xi_2\cdots \xi_{t}$, and for any choice of $S_{j_1}S_{j_2}$, $j_1, j_2=1,2$, we always have $|W^{P}R_i^{P}|\leq |vS_{j_1}S_{j_2}\xi_1\xi_2\cdots \xi_{t}|$, this gives a contradiction. Therefore $v$ has exactly two occurrences in $W^{P}R_i^{P}$.
\end{proof}


\begin{lemma}
\label{lem-2}
The identity
\begin{equation}
\label{eq-lem22}
\CT_{v}(\be)\geq \zeta(\be)\ga^{(Q-p)\be},
\end{equation}
where $Q$ is the greatest integer such that $\ga^{Q}< N$.

\end{lemma}
\begin{proof}

Since $|W|\geq |W^{P}|$, therefore from corollary \ref{cor-PS}, $W^{P}$ is a prefix of $W$. Also $|W|=F_p-2$, then there exists a $r\geq 1$ such that $p=P+r$, and from lemma \ref{cor-F-1}\eqref{eq1}, $W$ has the following decomposition $$W=W^{P}R_1^{P}R_1^{P+1}\cdots R_1^{P+r-1}.$$ Moreover, from Equality \eqref{eq-corF1} in Lemma \ref{cor-F-1} , the word $R_1^{P}R_1^{P+1}\cdots R_1^{P+r-1}$ can be decomposed in form of $R_{j_1}^{P}R_{j_2}^{P}\cdots R_{j_{m}}^{P}\in L(\Omega_{W^{P}})$, $j_{j^{\p}}\in \{1,2\}$, and $1\leq j^{\p}\leq m$ (see section \ref{sec-2}). Let us consider a word $$WR_{i_1}^{P}R_{i_2}^{P}\cdots R_{i_{k-1}}^{P},$$ and $R_{i_1}^{P}R_{i_2}^{P}\cdots R_{i_{k-1}}^{P}\in L(\Omega_{W^{P}})$. From corollary \ref{cor-PS}, the word $W^{P}$ is a prefix and a suffix of $W$. Therefore, the word $WR_{i_1}^{P}R_{i_2}^{P}\cdots R_{i_{k-1}}^{P}\in L(\om)$ if, and only if, $$W^{P}R_{j_1}^{P}R_{j_2}^{P}\cdots R_{j_{m}}^{P}R_{i_1}^{P}R_{i_2}^{P}\cdots R_{i_{k-1}}^{P}\in \CB_{_W}(m+k),$$ (see section \ref{sec-2.3}).

Furthermore, $W^{P}R_{i_1}^{P}R_{i_2}^{P}\cdots R_{i_{k-1}}^{P}$ will be a suffix of $WR_{i_1}^{P}R_{i_2}^{P}\cdots R_{i_{k-1}}^{P}$, and has exactly $k$ occurrences of $v$ ( see lemma \ref{lemWP}). \\ 

For each $k\geq 1$, we have $$|U^{(2)}(k)|\leq |WR_{i_1}^{P}R_{i_2}^{P}\cdots R_{i_{k-1}}^{P}|\leq F_{p}-2+(k-1)F_{P-1}.$$ Since $p\geq P$, thus $$|U^{(2)}(k)|+1\leq kF_{p}.$$

 
The inequality $$\CT_{v}(\be)\geq \sum_{k\geq 1}\Bigl({{N+1}\over {kF_{p}}}\Bigr)^{\be},$$ holds. Since $Q$ is the greatest integer such that $\ga^{Q}\leq N-1$.
Also $p$ is large, using the approximation $F_{p}\approx \ga^{p}$ (by Binet's formula), we get
$$\CT_{v}(\be)\geq\ga^{(Q-p)\be}\sum_{k\geq 1}\Bigl({{1}\over {k}}\Bigr)^{\be}= \zeta(\be)\ga^{(Q-p)\be}.$$ This completes the proof.
\end{proof} 





From \eqref{eq-9}, and \eqref{eq-2}, we have 

$$\l_{\be}\geq\sum_{|W|\geq N}\CC_{W}(\be)=e^{-NA}\sum_{|W|\geq N}\CT_{0}(\be)\CT_{W}(\be)\CT_v(\be).$$
 
Using lemma's $\ref{lem-1}$, $\ref{lem-2}$, and $\eqref{eq-T0}$, we get the following inequality;

$$\l_{\be}\geq e^{-NA} \zeta(\be) (\zeta(\be)-1) \bigl(\zeta(\be-1)-2\zeta(\be)\bigr)\sum_{p\geq P}\ga^{(Q-p)\be},$$ where $P$ is the least positive integer such that $\ga^{P}\geq N$. The series $\sum\limits_{p\geq P}\ga^{(Q-p)\be}$, is a geometric series with common ratio $\ga^{-\be}<1$, $\be>0$, therefore, $\sum\limits_{p\geq P}\ga^{(Q-p)\be}=\frac{\ga^{(Q-P)\be}}{1-\ga^{-\be}}$. We have $|W^{P}|\approx \ga^{P}$, and $|W^{Q}|\approx \ga^{Q}$. Also, no bispecial word has a length in $]\ga^Q, \ga^P[$, therefore from lemma \ref{cor-F-1}\eqref{eq-corF1}, $P=Q+1$, and $W^{P}=W^{Q}R_{1}^Q$. Hence 
$$\l_{\be}\geq e^{-NA} {(\zeta(\be)-1)}^{2} \bigl(\zeta(\be-1)-2\zeta(\be)\bigr)\frac{1}{\ga^{\be}-1},$$

with $\be>0$. 

If $\be_c\in\,\,]1,2]$, then $\l_{\be_c}=+\8$, which contradicts $\l_{\be_c}=1$ (see for instance, \cite[page. 25]{Bruin_2015}. Moreover, $\be\mapsto\l_{\be}$ is a decreasing function. Therefore, $\be_c \notin\, ]0,1]$. This completes the proof.


\section{Some Remarks.}
\label{sec-7}
Here a natural question arises, that why do we choose trajectories $\CD_{W}$. Assume that $W^{'}$ is another bispecial word, and $W\not= W^{'}$. Let $\CD_{W,W^{'}}$ be the collection of all possible trajectories satisfy (1)-(7) (in section \ref{sec-6}), but at time $l^{\p}$ the orbit of return word intersects the cylinder $[W^{'}]$ instead of $[W]$, and $W^{'}$ is the bispecial that appears at accident time $l^{\p}$ (see proposition \ref{prop-accident bispecial}). In this case, the possible sum of the contribution over the excursion zone will be:
\begin{equation}
\label{eq-remark}
\CE(\be)=\sum\limits_{\substack{U^{(0)}, U^{(1)}\\\,\, U^{(2)}\,\,possible}}{\Bigl(\frac{|W|+1}{|U^{(0)}|+1}\Bigr)}^{\be}{\Bigl(\frac{|W^{\p}|+1}{|U^{(1)}|+1}\Bigr)}^{\be}  {\Bigl(\frac{N+1}{|U^{(2)}|+1}\Bigr)}^{\be},
\end{equation}
(see \cite[section 4]{Bruin_2015}). The word $U^{(1)}$ in \eqref{eq-remark} is a non-identical bicephalic word for the pair $(W,W^{\p})$, and $U^{(2)}$ is the bicephalic word for pair $(W^{\p}, v)$. The identity:
\begin{equation}
\label{eq-r2}
\CT_{_{W, W^{\p}}}(\be)=\sum_{n\geq 1}\sum\limits_{\substack{U^{(1)}(n)\\\,\, U^{(1)}(n)\,\,non\,special}}\Bigl({{|W^{\p}|+1}\over {|U^{(1)}(n)|+1}}\Bigr)^{\be},
\end{equation}

will be the part of the total sum of the contribution of $\CD_{W,W^{'}}$ trajectories, that lies in the excursion zone. 
 The word $U^{(1)}(n)$ is a non-identical bicephalic word for the pair $(W,W^{\p})$, and $W^{\p}$ has exactly $n$ occurrences in $U^{(1)}(n)$ after the first occurrence of $W$. The word $U^{(1)}(n)$ can be decomposed as,
$$U^{(1)}(n)=U^{(1)}(1)R_{m_1}R_{m_2}\cdots R_{m_{n-1}},$$ where, $R_{m_1}R_{m_2}\cdots R_{m_{n-1}}\in L_{n-1}(\Omega_{_W^{\p}})$. \\

In the identical bicephalic word, the multiplicity will remain $n-2$ for $n\geq 2$ (see proposition \ref{prop-CB}) over the lower bound of identity $\CT_{W}(\be)$ (see the proof of lemma \ref{lem-1}). This is the reason that we can get a nice expression $$\zeta(\be-1)-2\zeta(\be),$$ in the total sum of the lower bound. Whereas in the non-identical bicephalic word, the difficulty will be yield by the word $U^{(1)}(1)$. The word $U^{(1)}(1)$ is a non-identical bicephalic word for the pair of bispecial words ($W, W^{\p}$), $(W\not=W^{\p})$ and $W^{\p}$ has precisely one occurrence in $U^{(1)}(1)$ after the occurrence of $W$ as a prefix. The significant difficulty arises, when the bispecial words $W$ and $W^{\p}$ overlapped, and the length of the bispecial words may affect the multiplicity of the possible number of overlapping. As in the Fibonacci language, the two different bispecial words always have different length.  That may change the multiplicity for any $n\geq 2$.  Furthermore, the sum in identity $\CT_{W, W^{\p}}(\be)$ will depend on $\ga^{p-q}$ , where $|W|\approx\ga^{p}$ and $|W^{\p}|\approx\ga^{q}$ with $p\not=q$. This makes the proof more technical and less easy to present. 



\cleardoublepage
\bibliographystyle{alpha}


\bibliography{/Users/shamsaishaq/Documents/MathTeX/Synopsis/reference.bib}



\end{document}